%% file: main.tex
\documentclass[12pt]{l4dc2020} 

\title[Training NN Controllers Using CBF in Presence of Disturbances]{Training Neural Network Controllers Using Control Barrier Functions in the Presence of Disturbances}
\usepackage{algorithm2e}
\newtheorem{Definition}{Definition}[section]
\newtheorem{Lemma}{Lemma}[section]

\author{%
 \Name{Shakiba Yaghoubi} \Email{syaghoub@asu.edu}\\
 \addr CIDSE, Arizona State University, Tempe, AZ, USA 
 \AND
 \Name{Georgios Fainekos} \Email{fainekos@asu.edu}\\
 \addr CIDSE, Arizona State University, Tempe, AZ, USA
 \AND
 \Name{Sriram Sankaranarayanan} \Email{srirams@colorado.edu}\\
 \addr Computer Science, University of Colorado Boulder, Boulder, CO, USA
}

\begin{document}

\maketitle
\vspace{-10pt}
\begin{abstract}%
Control Barrier Functions (CBF) have been recently utilized in the design of provably safe feedback control laws for nonlinear systems. 
These feedback control methods typically compute the next control input by solving an online Quadratic Program (QP).
Solving QP in real-time can be a computationally expensive process for resource constraint systems.
In this work, we propose to use imitation learning to learn  Neural Network based feedback controllers which will satisfy the CBF constraints.
In the process, we also develop a new class of High Order CBF for systems under external disturbances.
We demonstrate the framework on a unicycle model subject to external disturbances, e.g., wind or currents.
\end{abstract}


\section{Introduction}

Control Barrier Functions (CBF) have enabled the design of provable safe feedback controllers for a number of different systems such as adaptive cruise control \cite{ames2014control}, bipedal robot walking and long term autonomy \cite{ames2019control}.
CBF - along with Control Lyapunov Functions (CLF) - are typically part of the constraints of a Quadratic Program (QP) whose solution computes control inputs that guarantee safe system operation (while stabilizing to a desired operating point).
The CBF theory has been instrumental in developing safety critical controllers for nonlinear systems; however, it also has some limitations. 
First and foremost, it requires the online solution of a QP, which typically cannot satisfy hard real time constraints.
Second, the resulting controller may not be robust to noise and parameter or model inaccuracies, and to the best of our knowledge, robust high order control barrier functions have not been studied before.

In this work, we propose to use Neural Network (NN) based feedback controllers to address the aforementioned challenges.
Shallow NN-based controllers require limited memory and computational power and, therefore, they can address problem one.
In addition, NN-based controllers can be trained using both simulated and real data.
NN-based controllers can tolerate model uncertainty and inaccuracies.

In particular, in this work, we make the following contributions. 
First, we extend results from CBF \cite{ames2019control} and High Order CBF (HOCBF) \cite{xiao2019control} to nonlinear systems with affine controls and external disturbances.
Second, we adapt an imitation learning algorithm (\cite{ross2011reduction}) to train an NN-based controller from examples generated by the QP-based controller.
Even though in these preliminary results, we did not use real data, we demonstrate that we can train the NN controller robustly over non-noisy system trajectories and apply the resulting controller to a system subject to external disturbances. 
 Finally, even though in this paper, we do not address the provable safety of the resulting NN controller under all possible initial conditions, in the future, we plan to use tools like Sherlock \cite{DuttaEtAl2019hscc,DuttaEtAl2018adhs} to do so. 

{\bf Related Work:}
Input-to-state safety of a set $C$ which ensures that trajectories of a nonlinear dynamical system in presence of disturbances stay close to the set $C$, has been proved by enforcing the invariance of a larger set including $C$ in \cite{kolathaya2018input}.

The application of NN to control dynamical systems has a long history \cite{HuntEtAl1992automatica,HaganEtAl2002ijrnc,SchumannL2010book}.
More recently, due to computational advances and available data, there has been a renewed interest in the utilization of NN in control systems. 
\cite{yaghoubi2019worst} and \cite{claviere2019trajectory} utilize counterexample (adversarial sample) exploration to train NN that seek to satisfy a given property expressed either in temporal logic or through a reference trajectory.
The work by \cite{TuncaliEtAl2018dac} attempts to learn through simulations barrier certificates that can establish the safe operation of the closed loop system with an NN controller.
On the other hand,  \cite{zhang2019near} and \cite{chen2018approximating} take a different approach: they approximate Model Predictive Controllers (MPC) using supervised reinforcement learning for an NN.
Here, instead of approximating MPC, we approximate the solution of a QP constrained by HOCBF.

\input{body.tex}

\acks{This work was partially funded by NSF CNS 1932068, NSF IIP 1361926 and the NSF I/UCRC Center for Embedded Systems.}

\bibliography{Ref}

\end{document}

%% file: body.tex
\section{Preliminaries}
Consider a nonlinear control system without disturbances and with affine control inputs:
\begin{align}
    \label{eq1}
 \dot{x}=f(x)+g(x)u,
\end{align}
where $x\in X\subset \mathbb{R}^{n}$ is the system state,  $u\in U \subset \mathbb{R}^{m}$ is the control input, $f:\mathbb{R}^{n} \rightarrow \mathbb{R}^{n}$ and $g:\mathbb{R}^{n} \rightarrow \mathbb{R}^{m}$ are locally Lipschitz.  
A function $\alpha: \mathbb{R} \rightarrow \mathbb{R}$ is said to be an extended class $\mathcal{K}$ function iff $\alpha$ is strictly increasing and $\alpha(0) = 0$ (\cite{ames2019control}).


\begin{Definition}[Set Invariance \cite{blanchini1999set}]\label{def2}
A set $C \subseteq \mathbb{R}^n$ is forward invariant w.r.t the system (\ref{eq1}) iff  for every $x(0) \in C$, its solution satisfies $x(t)\in C$ for all $t \geq 0$. 
\end{Definition}

\begin{Definition}[Barrier Function]\label{defBF}
Let $h:X \rightarrow \mathbb{R}$ be a continuously differentiable function, $C: \ \{ x \in X | h(x) \geq 0 \}$ and $\alpha$ be a class $\mathcal{K}$ function.
 $h$ is a \emph{barrier function} iff
\begin{equation}\label{BF}
\dot{h}(x)\geq -\alpha(h(x))
\end{equation}
\end{Definition}

\begin{Lemma}[\cite{glotfelter2017nonsmooth}]\label{lemma1}
 If $h$ is a barrier function with $C$, and $\alpha$ as defined in Def.~\ref{defBF} then $C$ is a forward invariant set.
\end{Lemma}

\begin{Definition}[Control Barrier Function \cite{ames2019control}]
A continuous, differentiable function $h(x)$ is a Control Barrier Function (CBF) for the system (\ref{eq1}), if there exist a class $\mathcal{K}$ function $\alpha$ such that for all $x\in C$ :\vspace{-1pt}
\begin{equation}\label{CBF}
L_f h(x) +L_g h(x) u +\alpha(h(x))\geq 0 
\end{equation}
where $L_f h(x) = \frac{\partial h}{\partial x}^\top f(x), L_g h(x)= \frac{\partial h}{\partial x}^\top g(x)$ are the first order Lie derivatives of the system.
Any Lipschitz continuous controller $u \in K_{cbf}(x) = \{u\in U\;|\;L_f h(x) +L_g h(x) u +\alpha(h(x))\geq 0\}$ results in a forward invariant set $C$ for the system of Eq. (\ref{eq1}).
\end{Definition}
\begin{Definition}[Relative Degree of a Function]
A continuously differentiable function $h$ has a relative degree $m$ w.r.t the system (\ref{eq1}), if the first time that the control $u$ appears in the derivatives of $h$ along the system dynamics is in its $m^{th}$ derivative. 
\end{Definition}

If the function $h$ has a relative degree $m>1$, $L_g h(x) = L^{m-1}_g h(x) = 0$. As a result Eq. (\ref{CBF}) cannot be directly used for choosing safe controllers $u \in K_{cbf}(x)$. High Order Control Barrier Functions (HOCBF) were introduced in \cite{nguyen2016exponential}, and \cite{xiao2019control} to derive necessary conditions for guaranteeing the invariance of the set $C$. Assuming that the function $h$ has a relative degree $m$ w.r.t the system (\ref{eq1}), define the series of functions $ \psi_i:\mathbb{R}^n\rightarrow \mathbb{R}, i =1,\cdots,m$ and their corresponding sets $C_1,\cdots,C_m$ as follows:  
\begin{align}\label{series}
       \psi_0 (x)& = h(x)\qquad \qquad \qquad \qquad \qquad \qquad C_1 = \{x\;|\; \psi_0 (x)\geq0\}\nonumber\\
      \psi_1 (x)& = \dot{ \psi}_0 (x)+\alpha_1( \psi_0 (x)) \qquad \qquad \qquad  C_2 = \{x\;|\; \psi_1 (x)\geq0\}\\[-7pt]
      \vdots&\qquad \qquad \quad \qquad \qquad \qquad \qquad \qquad \qquad\vdots \nonumber\\[-7pt]
      \psi_m (x)& = \dot{ \psi}_{m-1} (x)+\alpha_m( \psi_{m-1} (x)) \qquad \quad  C_m = \{x\;|\; \psi_{m-1} (x)\geq0\}\nonumber
\end{align}
where $\alpha_1,\alpha_2\cdots, \alpha_m $ are class  $\mathcal{K}$ functions of their arguments. 
\begin{Definition}[High Order Barrier Functions]
A function $h:\mathbb{R}^n\rightarrow \mathbb{R}$ with a relative degree $m$ is a High Order Barrier Function (HOBF) for system (\ref{eq1}), if there exist differentiable class $\mathcal{K}$ functions $\alpha_1,\alpha_2\cdots, \alpha_m $ such that for all $x\in C_1\cap C_2\cap\cdots \cap C_m$, we have:   $\psi_m (x)\geq 0$. Under this condition, the set $C_1\cap C_2\cap\cdots \cap C_m$ is forward invariant.
\end{Definition}
\begin{Definition}[High Order Control Barrier Functions \cite{xiao2019control}]
A function $h: \mathbb{R}^n\rightarrow \mathbb{R}$ with a relative degree $m$ is a High Order Control Barrier Function (HOCBF) for system (\ref{eq1}), if there exist differentiable class $\mathcal{K}$ functions $\alpha_1,\alpha_2\cdots, \alpha_m $ such that for all $x\in C_1\cap C_2\cap\cdots \cap C_m$:
\begin{equation}\label{HOCBF}
    L_f^m h(x)+L_g L_f^{m-1} h(x) u + O(h(x))+ \alpha_m(\psi_{m-1} (x))\geq 0
\end{equation}
where $O(.)$ denotes the remaining Lie derivatives
along $f$ with degree less than or equal to $m-1$. Any controller $u \in K_{hocbf}(x) = \{u\in U \;|\;  L_f^m h(x)+L_g L_f^{m-1} h(x) u + O(h(x))+ \alpha_m(\psi_{m-1} (x))\geq 0\}$ renders the system safe, and the set $C_1\cap C_2\cap\cdots \cap C_m$ forward invariant.
\end{Definition}
\section{Control Barrier Functions in presence of Disturbance}

In this paper, the nonlinear control system (\ref{eq1}) is considered in presence of disturbances:
\begin{align}
    \label{eq:dist}
 \dot{x}=f(x)+g(x)u+Mw, \quad x(0)\in X_0
\end{align}
where $x,u,f,g$ are defined as for the system of Eq. (\ref{eq1}), $X_0$ is the set of initial conditions, $w\in W\subset \mathbb{R}^{l}$ is the disturbance input, each dimension of $W$ which we denote by $W_i$ defines an interval  $[\underline{w_i},\overline{w_i}]$ that the $i^{th}$ element of w belongs to, $M$ is a $n\times l$ zero-one matrix with at most one non-zero element in each row. We assume that $w$ is Lipchitz continuous. If $u$ is also Lipchitz, the solutions $x(t), t>0$ to the system (\ref{eq:dist}) are forward complete.

When disturbance is present, in order to guarantee the forward invariance of the set $C$, which we call the safe set, the condition in inequality (\ref{BF}) needs to be satisfied for all $w\in W$, including its worst case where it minimizes the left hand side of the inequality.

\begin{Definition}\label{PCBF}
The continuously differentiable function $h$ of relative degree one is a CBF in presence of Disturbance (CBFD) for the system of Eq. (\ref{eq:dist}), if there exist a class $\mathcal{K}$ function $\alpha$ such that for all $x\in C$ and $w\in W$, inequality \ref{jh} is satisfied or equivalently for all $x\in C$ inequality \ref{ll} holds in which $L_M h(x) = \frac{\partial h}{\partial x}^\top M$, and $w_{opt} =  \underset{w\in W}{\arg\max} \;(-L_M h(x) w)$
\begin{equation}\label{jh}
L_f h(x) +L_g h(x) u +L_M h(x) w+\alpha(h(x))\geq 0 
\end{equation}
\begin{equation}\label{ll}
L_f h(x) +L_g h(x) u +\alpha(h(x))\geq -L_M h(x) w_{opt} 
\end{equation}

\end{Definition}
Since $L_M h(x) w$ is linear in $w$, and $w\in W$ imposes linear constraints on $w$, $\max_{w\in W}(-L_M h(x) w)$ is a linear program for each $x\in C$ whose solution can be found and replaced in inequality (\ref{ll}) to define the set of control values that satisfy the following inequality:
\begin{equation}
   K_{cbfd}(x) = \{u\in U\;|\;L_f h(x) +L_g h(x) u +\alpha(h(x))\geq -L_M h(x) w_{opt} \}
\end{equation}\vspace{-8pt}
\begin{theorem}
Given a CBFD $h$ from Def. (\ref{PCBF}), any Lipchitz continuous controller $u \in K_{cbfd} (x)$ renders the set $C$ forward invariant.
\end{theorem}
\begin{proof}
The proof can be directly derived from Lemma \ref{lemma1}. To be explicit, if for all $x\in C$ and $w\in W$, $\dot {h}(x) = L_f h(x) +L_g h(x) u +L_M h(x) w \geq -\alpha(h(x))$, then the solutions to system (\ref{eq:dist}) with $x(0) \in C$, satisfy $h(x(t))\geq 0$. So based on Def. \ref{def2}, $C$ is forward invariant. 
\end{proof}

If the function $h$ has a relative degree higher than one, the multiplier of $u$ in Eq. (\ref{ll}), $L_g h(x)$ is equal to zero, so the  choice of $u$ will not affect the satisfaction of inequality. (\ref{ll}). In the following section we will study HOCBFs in presence of disturbance.


\section{High order Control Barrier Functions in Presence of Disturbance} 
Assume that the continously differentiable function $h:\mathbb{R}^n\rightarrow \mathbb{R}$ has relative degree $m$ and consider the series of functions $\psi_i:\mathbb{R}^n\rightarrow \mathbb{R}, i =1,...,m$ and their corresponding sets $C_1, \ldots, C_m$ as defined in Eq. (\ref{series}). 

\begin{Definition}
The function $h$ is a High Order Barrier Function in presence of disturbance (HOBFD) for system (\ref{eq:dist}), if there exist differentiable class $\mathcal{K}$ functions $\alpha_1,\alpha_2, \ldots, \alpha_m $ that define the functions $\psi_1,\cdots,\psi_m$, such that for all $x\in C_1\cap C_2\cap\cdots \cap C_m$, we have:  $\psi_m (x)\geq 0$
\end{Definition}

\begin{Definition}\label{sec:HOCBFD}
The function $h$ is a High Order Control Barrier Function in presence of disturbance (HOCBFD) for system (\ref{eq:dist}), if there exist differentiable class $\mathcal{K}$ functions $\alpha_1,..., \alpha_m $ that define the functions $\psi_1,...,\psi_m$, s.t for all $x\in C_1\cap C_2\cap... \cap C_m$ and $w\in W$:
\begin{equation}\label{HOCBFD}
\psi_m (x)=L_f^m h(x)+L_g L_f^{m-1} h(x) u + P(x,w) +O(h(x))+ \alpha_m(\psi_{m-1} (x))\geq 0
\end{equation}
where $P(x,w)$ is a function of the $x$ and $w$, $O(.)$ denotes the remaining Lie derivatives along $f$ with degree less than or equal to $m-1$.
Since equation (\ref{HOCBFD}) needs to be satisafied for all $w\in W$, we can equivalently write it as:
\begin{equation}\label{HOCBD2}
L_f^m h(x)+L_g L_f^{m-1} h(x) u +O(h(x))+ \alpha_m(\psi_{m-1} (x))\geq  -P(x,w_{opt})
\end{equation}
where $w_{opt} =  \underset{w\in W}{\arg\max} \;(-P(x,w))$.
\end{Definition}

The set of control inputs that satisfy inequality (\ref{HOCBD2}) is:
\begin{equation*}
   K_{hocbfd}(x) = \{u\in U\;|\;L_f^m h(x) +L_g L_f^{m-1} h(x) u +O(h(x))+ \alpha_m(\psi_{m-1} (x))\geq  -P(x,w_{opt})\}
\end{equation*}

\begin{theorem}
Given a HOCBFD $h$ from Def. (\ref{sec:HOCBFD}), any Lipchitz continuous controller $u \in K_{hocbfd} (x)$ renders the set  $C_1\cap C_2\cap\cdots \cap C_m$ forward invariant.
\end{theorem}
\begin{proof}
Any Lipchitz controller $u \in K_{hocbfd} (x)$ enforces $\psi_m (x)\geq 0$ or equivalently $\dot{\psi}_{m-1} (x)\geq -\alpha_m({\psi}_{m-1}(x))$, irrespective of the value of $w\in W$.  Assuming that $x(0) \in C_1\cap C_2\cap\cdots \cap C_m$, and hence $x(0) \in C_{m}$, we have ${\psi}_{m-1}(x(0))\geq0$ which based on lemma \ref{lemma1}, lead to  ${\psi}_{m-1}(x)\geq0$ ($x\in C_{m}$) or equivalently 
$\dot{\psi}_{m-2} (x)\geq -\alpha_{m-1}({\psi}_{m-2}(x))$, again since  $x(0) \in C_{m-1}$ this results in ${\psi}_{m-2}(x)\geq 0$ ($x\in C_{m-1}$). Continuing this reasoning, we can prove that $C_1\cap C_2\cap\cdots \cap C_m$ is forward invariant.
\end{proof}
\begin{remark}
In order to use HOCBFDs to prove that all the trajectories of the system \ref{eq:dist} starting from $X_0$ will never exit $C_1$, the sets $C_1, C_2,\cdots, C_m $ should have a nonempty interior, and the set of initial conditions of the system, $X_0$, should be a subset of $C_1\cap C_2\cap\cdots \cap C_m$. Note that if $X_0 \subset C_1$ ($h(x(0)) \geq 0$), except for special cases (see \cite{xiao2019control}) which we do not consider here, we can always choose $\alpha_1,\alpha_2\cdots, \alpha_m $ such that $x_0\in C_2\cap\cdots \cap C_m$.
\end{remark}
Note that the problem $\underset{w\in W}{\mbox{max}} \;(-P(x,w))$, is in general a nonlinear program and finding its optimal - or even suboptimal - solution can be time consuming. A special case of the problem is if we consider the linear class $\mathcal{K}$ functions $\alpha_1,\cdots, \alpha_m $ which will form Exponential Control Barrier Functions \cite{ames2019control}. This makes $P(x,w)$ a polynomial function of degree $m$ in $w$. In case of polynomial functions $\alpha_1,\cdots, \alpha_m$, $P(x,w)$ will be a polynomial function of $w$ - potentially of higher degree than $m$.

When $m = 2$, and $\alpha_1,\cdots, \alpha_m$ are linear functions, $P_m(x,w)$ is a quadratic function of $w$, and  $\underset{w\in W}{\mbox{max}} \;(-P_m(x,w))$ is a QP for each $x\in X$ that can be solved efficiently. 

\begin{example}
Consider the system $\dot{x}_1=x_2+w$, $\dot{x}_2=u$ with $ w\in [\underline{w},\overline{w}]$. The control input should be designed such that the function $h(x) = x_1^2-1$ is a HOCBFD. We consider $\alpha_i (y)= y,i = 1,2$, so we have $\alpha'_i (y)= \frac{\partial \alpha_i}{\partial y} = 1$, and as a result: 
\begin{align*}
        \psi_2 (x) &= \ddot{h}(x)+\alpha_1'(h(x))\dot{h}(x)+\alpha_2(\dot{h}(x)+\alpha_1(h(x))) \\
        &= 2x_1 u+\underbrace{(4x_2+4x_1)w+2w^2}_{P(x,w)}+2x_2^2+4x_1x_2+x_1^2-1  
\end{align*}
observing that $w_{opt} =  \underset{\underline{w}<w<\overline{w}}{\arg\max} \;(-2w^2-(4x_2+4x_1)w)$ is a quadratic program that can be solved at each $x$, any Lipchitz controller in the set $K_{hocbfd}(x) = \{2x_1 u+2x_2^2+4x_1x_2+x_1^2-1>  -2w_{opt}^2-(4x_2+4x_1)w_{opt}\}$ will make $C = \{x\;|\;h(x)\geq0\}$ forward invariant.

\end{example}
\section{Control Optimization Problem with CBF constraints}\label{sec5}
In order to find safe sub-optimal controllers, many recent works \cite{lindemann2019control,xiao2019control,ames2014control, yang2019sampling}, formulate optimization problems with quadratic costs in the control input $u$ subject to CLF and CBF constraints (each CBF constraints corresponds to an unsafe set) which are linear in $u$. These QPs are solved every time new information about the states $x$ are received, and the resulting control value $u$ is used in the time period before new information is received. In presence of disturbances, in order to formulate the QPs with constraints of type (\ref{ll}) or (\ref{HOCBD2}), $w_{opt}$ should be computed as a prerequisite. To compute $w_{opt}$ one need to solve $\underset{w\in W}{\mbox{max}} \;(-L_M h(x) w)$ or $\underset{w\in W}{\mbox{max}} \;(-P(x,w))$ - depending on the relative degree $m$ - for each barrier function or unsafe set. As a result, formulating the quadratic program and solving it for evaluating the control input $u$ may not be possible at run-time. In the following section, we present a paradigm for training NN controllers that predict the value of the control input resulting from the quadratic programs. 

\section{Learning NN Controllers from Control Barrier Functions Using the DAGGER Algorithm} 
Imitation learning methods, which use expert demonstrations of good behavior to learn controllers, have proven to be very useful in practice \cite{ho2016generative, abbeel2004apprenticeship,bagnell2007boosting,reddy2019sqil,song2018multi}. While a typical method to imitation learning is to train a classifier/regressor to predict an expert's behavior given data from the encountered observations and expert's actions in them, it's been shown in \cite{ross2011reduction} that using this framework,  small errors made by the learner can lead to large errors over time. The reason is that in this scenario, the learner can encounter completely different observations than those it has been trained with, leading to error accumulation. Motivated by this, \cite{ross2011reduction} presents an algorithm called DAGGER (Dataset Aggregation) that iteratively updates the training dataset with new observations encountered by the learner and their corresponding expert's actions and retrains the learner.  

As described in Section \ref{sec5}, forming and solving the required quadratic programs may not be feasible at run-time. As a result, we use an algorithm inspired by the DAGGER algorithm to train NN controllers that predict the outcome of the quadratic program.
In this regard, the QP acts as an expert that a NN imitates. An NN controller that has been trained offline can be used in a feedback loop to produce the desired control values online. The NN training algorithm is described in Alg. \ref{alg} in which it is assumed that $\pi^*(x)$ is an expert that performs the QP routine at $x$ to output the desired control value. 

\begin{algorithm2e}[h]
\small{ \KwData{The dynamical system \ref{eq:dist}, $W$, the set of initial conditions $X_0$, the constant $0<p<1$, maximum number of iterations $N$} 
 Randomly choose the set $X_0^s$ by sampling from $X_0$ \;
Sample trajectories of the system \ref{eq:dist} with initial conditions in $X_0^s$ and input $\pi_0 = \pi^*(x)$\;
 Initialize $D$ with the pairs of visited states and corresponding control inputs: $D = {(x, \pi^*(x))}$\;
 Train NN controller $\hat{\pi}_1$ on $D$\;
     \For{$i = 1,...,N$}{
     $\beta = p^{i}$\;
     Sample trajectories of the system \ref{eq:dist} with $x(0)\in X_0^s$ and input $\pi_i=\beta\pi^*(x)+(1-\beta)\hat{\pi}_i(x)$\;
      Get dataset $D_i = {(x, \pi^*(x))}$ of  visited states and corresponding control inputs\;
      Aggregate datasets: $D \leftarrow D\cup D_i$\;
      Train NN controller $\hat{\pi}_i$ on $D$\;
  }}
 \Return the best $\hat{\pi}_i$ on validation; 
\caption{Data set Aggregation for training NN using Quadratic Programs}\label{alg}
\end{algorithm2e}\setlength{\textfloatsep}{7pt}

\section{Reach Avoid Problem of a Water Vehicle Model}
Consider the model of a surface water vehicle subject to wind gusts and water currents as:
\begin{equation}\label{uni-model}
  {\small  \dot{x} = \begin{bmatrix}
     \dot{x_1}\\
     \dot{x_2}\\
     \dot{\theta}
    \end{bmatrix}= \begin{bmatrix}
    v \cos(\theta)\\
    v \sin(\theta)\\
    0 
    \end{bmatrix}+\begin{bmatrix}
    0\\
    0\\
    1 
    \end{bmatrix}u+\begin{bmatrix}
    1\\
    1\\
    0 
    \end{bmatrix}w, \qquad x(0) \in X_0}
\end{equation}
where the state $x\in\mathbb{R}^3$ consists of vehicle location $(x_1,x_2)$ and the heading angle $\theta$. The control input $u\in \mathbb{R}$ is the vehicle's steering angle.
The velocity $v$ is assumed to be constant ($v = 1$) as it has a different relative degree from the steering angle $u$\footnote{Considering $v$ as an input  will make CBF constraints nonlinear in $v$, and the resulting problem will not be a quadratic program anymore. While this nonlinear program can be solved offline in this framework, in this paper we assume $v$ is constant for simplicity.}. The external disturbance is $w\in[-0.1,0.1]$. System trajectories starting from the set $X_0 = [8, 9]\times[5, 11]\times[-\pi, \pi]$ should avoid the unsafe sets $\mathcal{U}_i, i = 1,...,5$ and reach the goal set $\mathcal{G}$:
\begin{align*}
    \mathcal{U}_i =\{x: (x_1-p_i(1))^2+(x_2-p_i(2))^2<r_i\},\quad
\mathcal{G} = \{x: (x_1-x_{g,1})^2+(x_2-x_{g,2}))^2<0.3\}
\end{align*}
where $p_1 = (4, 2.5), r_1= 0.7, p_2 = (5, 6.5), r_2 = 0.5, p_3 =(7, 4.75), r_3 = 0.4, p_4 = (2.5, 5), r_4 = 0.3, p_5= (7.5, 2.5), r_5 = 0.5$, and $x_{g,1} = x_{g,2} = 1$.

In order to reach the goal set, instead of using CLF based constraints, we formulate the stabilizing condition in the objective function. The desired heading angle is $\theta_{ref}(x) = \mbox{arctan}(\frac{x_{g,2}-x_2}{x_{g,1}-x_1})$, and the desired input $u$ to force $\theta$ to follow $\theta_{ref}$ is $u_{ref}(x) = K(\theta_{ref}(x)-\theta)$ where $K$ is a positive constant, here we choose $K=1$. The barrier function corresponding to the unsafe set $\mathcal{U}_i$ is $h_i (x) =   (x_1-p_i(1))^2+(x_2-p_i(2))^2-r_i$ which has relative degree 2 w.r.t to the steering angle $u$. We consider $\alpha_1 (y)= \alpha_2(y) = 2y$. The function $\dot{\psi}_{2,i}$ corresponding to each $h_i$ can be computed based on Eq. (\ref{series}) using Matlab's Symbolic toolbox, for example:
\begin{align*}\vspace{-3pt}
   \dot{\psi}_{2,1} =& -(2\sin(\theta)(x_1 - 4) - 2\cos(\theta)(x_2 - 2.5))u  \qquad\leftarrow L_g L_f h_1(x) u\\
   &+4w^2+ 4(\cos(\theta)+\sin(\theta)+2((x_1-4)+(x_2-2.5)))w\qquad\leftarrow P_1(x,w)\\
   & +4(x_1 - 4)^2 + 4(x_2 - 2.5)^2 + 4(\cos(\theta))(2x_1 - 8) + 4(\sin(\theta))(2x_2 - 5) - 0.8
\end{align*}\vspace{-3pt}

The functions $P_i(x,w)$ corresponding to each unsafe set are quadratic in $w$. Take $w_{opt,i} (x)= \underset{-0.1<w<0.1}{\arg\max} (-P_i(x,w))$ which needs to be solved for each unsafe set at each state. Also, let's call the portion of $\dot{\psi}_{2,i}$ that only depends on $x$,  $\Psi_{i}$. Note that $\Psi_{i} (x)= L_f^2 h_i(x)+O(h(x))+ \alpha_2(\psi_{1} (x)) $. As a result, in order to reach the goal set while avoiding the unsafe sets, the following quadratic program needs to be solved:
\begin{align}\label{uni-QP}
   & \min_{u} \quad (u-u_{ref}(x))^2\\
    s.t  \quad
 L_g L_f h_i(x) &u+\Psi_{i}(x)\geq  -P(x,w_{opt,i}(x)) \quad\forall i = 1,...,5\nonumber
\end{align}
\begin{figure}[t]
  \centering
  \subfigure[]{\includegraphics[scale=0.65]{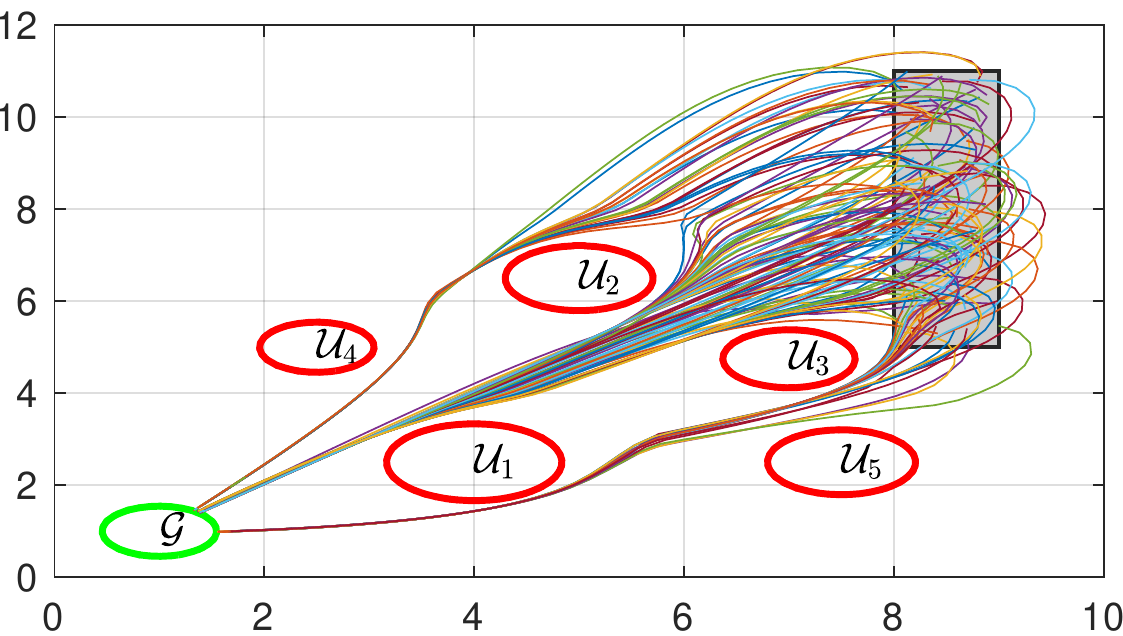}}\vspace{-5pt}
  \subfigure[]{\includegraphics[scale=0.65]{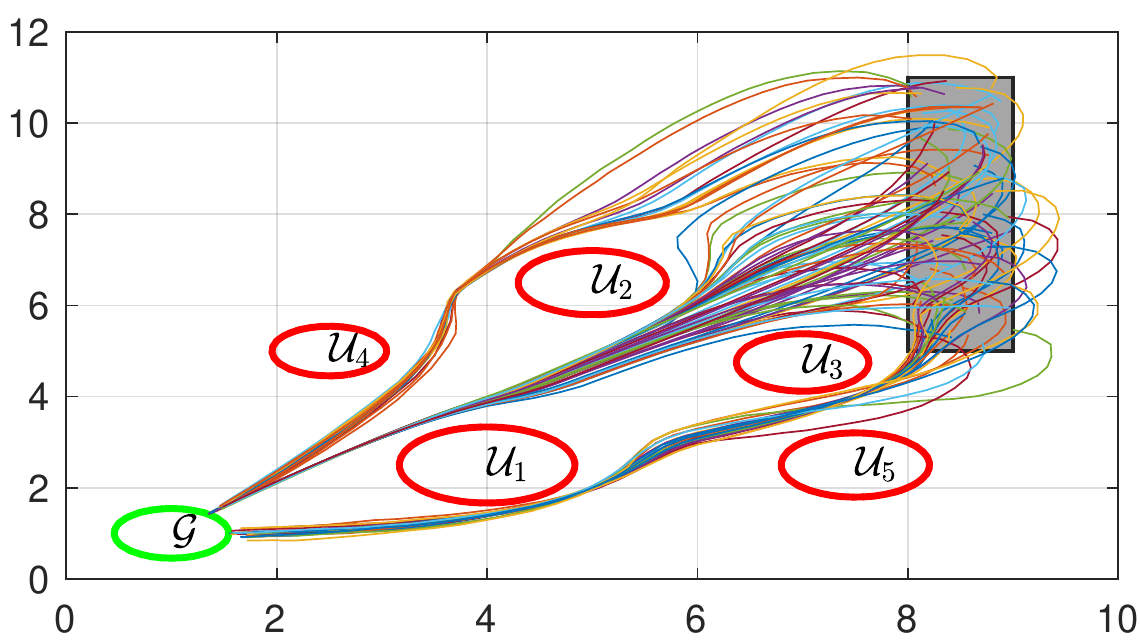}}\vspace{-5pt}
  \caption{Trajectories initiated from $X_0^s$ as guided by ({\it a}) the QPs as expert when $w=0$ and ({\it b}) the trained NN controller when random disturbance is applied to system}\vspace{-13pt}
  \label{fig}
\end{figure}
This QP is solved at each state visited by the vehicle under the controller $\pi_i$ until reaching the goal set $\mathcal{G}$, as described in Alg. \ref{alg} to train NN controllers that can predict the expert's action online. Figure \ref{fig}.({\it a}) shows the trajectories of the system (\ref{uni-model}) guided by the solutions to QPs in Eq. (\ref{uni-QP}) when $w = 0$. The NN controller successfully imitates the QPs at the $11^{th}$ iteration of the for loop in Alg. \ref{alg}. Figure \ref{fig}.({\it b}) shows the system trajectories guided by the trained NN controller when randomized disturbance is applied to the system. As it is clear from the figures the controller is robust to disturbances as it has been trained with controllers that are able to compensate for the disturbance in the worst-case. It is worth mentioning that the inputs to the NN are the location states $(x_1,x_2)$ in addition to $(\sin(\theta),\cos(\theta))$ - instead of the state $\theta$ itself. This data processing helps remove the discontinuities than happen when mapping $\theta$ to $[-\pi,\pi]$ and helps NN understand that $-\pi$ and $\pi$ are indeed equivalent.
Also, even-though input constraints are not enforced in this example, they can be added to problem (\ref{uni-QP}) as linear constraints and considered in the NN architecture by adding a saturation function in the output.


\section{Conclusions}
In this work, we studied Control Barrier Functions (CBF) in presence of disturbances. 
These functions define constraints on the control input that can be used in an optimization problem to find safe sub-optimal control inputs. 
As solving these optimization problems might not be possible in real-time, we presented a framework to train NN controllers that can be used online to predict the outcome of the optimization problems.
Future work will use methods like \cite{DuttaEtAl2018adhs} to establish safety of the learned controller and counter-example generation methods as in \cite{YaghoubiF2019hscc} to speed up training. 
\vspace{-2pt}